\documentclass[11pt]{amsart}
\pdfoutput=1
\usepackage[T1]{fontenc}
\usepackage[utf8]{inputenc}
\usepackage[english]{babel}
\usepackage{amsmath, amssymb}
\usepackage{graphicx}
\usepackage{amsthm}
\usepackage{array}
\usepackage{amsthm}
\usepackage{mathtools}
\usepackage{thmtools}
\usepackage{xspace}
\usepackage[margin=.85 in]{geometry}
\allowdisplaybreaks

\usepackage[dvipsnames]{xcolor}
\usepackage{hyperref}
\definecolor{candyapplered}{rgb}{1.0, 0.03, 0.0}
\definecolor{mediumblue}{rgb}{0.0, 0.0, 0.8}

\hypersetup{
    bookmarksnumbered,
    unicode=false,          % non-Latin characters in Acrobat’s bookmarks
    pdftoolbar=true,        % show Acrobat’s toolbar?
    pdfmenubar=true,        % show Acrobat’s menu?
    pdffitwindow=false,     % window fit to page when opened
    pdfstartview={FitH},    % fits the width of the page to the window
    pdftitle={The sum of Lagrange numbers},    % title
    pdfauthor={Jonah Gaster and Brice Loustau},     % author
    pdfnewwindow=true,      % links in new PDF window
    colorlinks=true,       % false: boxed links; true: colored links
    linkcolor=mediumblue,          % color of internal links (change box color with linkbordercolor)
    citecolor=candyapplered,        % color of links to bibliography
    filecolor=magenta,      % color of file links
    urlcolor=blue          % color of external links
}

\usepackage{caption,subcaption}
\RequirePackage[ocgcolorlinks]{ocgx2} % So that hyperref links are printed in black

\declaretheorem[numberwithin=section]{theorem}

\declaretheorem[sibling=theorem]{lemma}
\declaretheorem[sibling=theorem, style=remark]{remark}

\newcommand*{\ie}{i.e.\@\xspace}
\makeatletter
\newcommand*{\eg}{%
    \@ifnextchar{.}%
        {e.g.}%
        {e.g.\@\xspace}%
}
\newcommand*{\etc}{%
    \@ifnextchar{.}%
        {etc}%
        {etc.\@\xspace}%
}
\makeatother

\addto\extrasenglish{

}
\addto\extrasfrench{

}

\newcommand{\bZ}{\mathbb{Z}}
\newcommand{\bP}{\mathbb{P}}
\newcommand{\bN}{\mathbb{N}}

\newcommand{\Mod}{\mathrm{Mod}}
\newcommand{\Aut}{\mathrm{Aut}}
\newcommand{\bR}{\mathbb{R}}

\newcommand{\PGL}{\mathrm{PGL}}

\newcommand{\cL}{\mathcal{L}}
\newcommand{\cM}{\mathcal{M}}

\newcommand{\Fix}{\mathrm{Fix}}
\newcommand{\cS}{\mathcal{S}}

\newcommand{\upd}{\mathop{}\!\mathrm{d}}
\DeclareMathOperator{\sech}{sech}

\mathtoolsset{showonlyrefs=true}
\newcommand{\MUC}{MUC}

\begin{document}
\title[The sum of Lagrange numbers]{The sum of Lagrange numbers}
\author{Jonah Gaster and Brice Loustau}
\date{August 17, 2020}

\address{Department of Mathematical Sciences, University of Wisconsin-Milwaukee}
\email{gaster@uwm.edu}

\address{Mathematisches Institut, Universit\"at Heidelberg ~and~ Heidelberg Institute of Theoretical Studies}
\email{brice.loustau@uni-heidelberg.de}

\begin{abstract}
Combining McShane's identity on a hyperbolic punctured torus with Schmutz's work on the Markov Uniqueness Conjecture (\MUC{}),
we find that \MUC{} is equivalent to the identity 
\begin{equation}
\sum_{n=1}^\infty \, \left( 3- L_n \right) \, = \, 4 - \varphi - \sqrt 2
\end{equation}
where $L_n$ is the $n$th Lagrange number and $\varphi=\frac{1+\sqrt5}2$ is the golden ratio.
\end{abstract}

\maketitle

\section{Preliminaries}

\subsection{Lagrange and Markov numbers}

The \emph{Lagrange numbers} $\cL=\{L_n\}_{n=1}^\infty=\{\sqrt 5, \sqrt 8 , \ldots\}$ are a sequence of real numbers that naturally arise in Diophantine approximation. Hurwitz's theorem states that for any irrational number $x$, there exists a sequence of rationals $p_n/q_n$
converging to $x$ with
$\left| x - \frac{p_n}{q_n}\right| < \frac{1}{\sqrt{5} q_n^2}$. 
In this expression, $\sqrt{5}$ is optimal, as can be shown by taking $x = \varphi$ (the golden ratio). 
It turns out that when $x = \varphi$ and related numbers are excluded, $\sqrt{8}$ is the new best constant. 
By definition, $L_1 = \sqrt{5}$ is the first Lagrange number, $L_2 = \sqrt{8}$ is the second Lagrange number, \etc{}

The \emph{Markov numbers} $\cM =\{m_n\}_{n=1}^\infty = \{1, 2, 5, 13, \ldots\}$ are the positive integers that appear in a Markov triple, \ie{} a solution $(x,y,z)\in \bZ^3$ to the cubic 
\begin{equation} \label{eq:Markov}
x^2+y^2+z^2=3xyz\,.
\end{equation}
In 1880, Markov \cite{Markoff, MR1510073} discovered a remarkable connection between this cubic and the theory of binary quadratic forms, and proved the unexpected relation between Markov and Lagrange numbers:
\begin{equation} \label{eq:defLagrange}
 L_n = \sqrt{9 - \frac{4}{m_n^2}}\,.
\end{equation}

Using the Vieta involution $(x,y,z) \mapsto (x,y,3xy-z)$, it is easy to see that for any Markov number $m$, one can always find a Markov triple $(x,y,m)$ with $0 < x \leqslant y \leqslant m$. The \emph{Markov Uniqueness Conjecture} (\MUC{}) asserts that such a triple is always unique. \MUC{} was initially offered by Frobenius in 1913 \cite{frobenius} and is notoriously difficult \cite{Guy}. 
For more context and detail, we refer to \cite{Aigner, cusick-flahive}.

\subsection{The sum of Lagrange numbers} It is clear from \eqref{eq:defLagrange} that $L_n$ is an increasing sequence of positive numbers that converges to $3$ when $n \to +\infty$. Moreover, we have $3 - L_n \sim \frac{2}{3m_n^2}$, and since $m_n \geqslant n$ (actually $m_n$ is much greater, see \autoref{sec:Numerical}), the series $\sum_{n=1}^\infty (3 - L_n)$ is convergent. In this paper, we prove: 
\begin{theorem} \label{thm:MainThm}
The Markov Uniqueness Conjecture holds if and only if 
\begin{equation}
\label{eq:LagrangeSum}
\sum_{n=1}^\infty \, \left( 3- L_n \right) \, = \, 4 - \varphi - \sqrt 2\,.
\end{equation}
\end{theorem}

The proof is easily derived from the McShane identity on a hyperbolic punctured torus 
and a result of Schmutz regarding the well-known relationship between hyperbolic geometry and Markov numbers.
It is nonetheless a striking identity, and could optimistically open a new path towards probing \MUC{}.

\begin{remark}
Several authors have explored similar ideas, for instance \cite{MR1356829}, \cite[\S 4.3]{Lang-Tan}. 
\end{remark}

\begin{remark}
Numerical computation confirms the identity \eqref{eq:LagrangeSum} convincingly, as we shall see in \autoref{sec:Numerical}. This is not surprising since \MUC{} has also directly been checked by computers for high values of $n$.
\end{remark}

\subsection{Markov numbers and the modular torus.} \label{subsec:RelationModularTorus}
The beautiful relationship between Markov numbers and hyperbolic geometry was discovered by Gorshkov \cite{Gorshkov} and Cohn \cite{MR67935}. Let $T^*$ denote the once-punctured torus, \ie{} the topological surface obtained by removing a point from the torus $T^2$. For a certain hyperbolic metric on $T^*$, the lengths of simple closed geodesics on $T^*$ are given by the Markov numbers. We briefly explain this connection and refer to \eg{} \cite{Series} for more discussion.

The \emph{character variety} of the once-punctured torus is the cubic surface $\mathcal{X}$ defined by the equation
\begin{equation} \label{eq:CharacterVariety}
 x^2 + y^2 + z^2 = xyz\,.
\end{equation}
Hyperbolic metrics on $T^*$ with finite volume correspond to real points of ${\mathcal X}$. Indeed, let $\pi_1(T^*) = \langle a, b\rangle$ where $a$ and $b$ are the standard generators of $\pi_1(T^2) \approx \bZ^2$. Hyperbolic structures on $T^*$ are parametrized by $x = \operatorname{tr}(A)$, $y = \operatorname{tr}(B)$, $z = \operatorname{tr}(AB)$ where $A, B \in \operatorname{SL}_2(\bR)$ are (lifts of) the holonomies of $a, b \in \pi_1(T^*)$. The condition that the metric has finite volume amounts to the peripheral curve $aba^{-1}b^{-1}$ having parabolic holonomy, \ie{} $\operatorname{tr}(ABA^{-1}B^{-1}) = -2$. Using the classical trace relations in $\operatorname{SL}_2(\bR)$, this equation is rewritten $ x^2 + y^2 + z^2 = xyz$. We refer to \eg{} \cite{Goldman} for more details on this correspondence.

The integer solutions $(x,y,z) \in \bZ^3$ of \eqref{eq:CharacterVariety} are clearly in bijection with Markov triples: $x, y, z$ must all be divisible by $3$, and the reduced triple $(x/3, y/3, z/3)$ verifies \eqref{eq:Markov}. 
Thus Markov triples are the integral points of $\mathcal{X}$ (up to $1/3$). 
In fact, the mapping class group $\Mod(T^*)$ acts transitively on such triples, \ie{} all corresponding hyperbolic tori are isometric. This hyperbolic torus is called the \emph{modular torus} $X$, a 6-fold cover of the modular orbifold. 
Markov numbers can alternatively be described as one third of traces of simple closed geodesics on $X$:
\begin{equation}
 3 \, \cM = \left\{3 \,m_n, \,n \in \bN\right\} = \left\{\tau(\gamma) , \,\gamma \in \cS\right\}
\end{equation}
where we denote $\cS$ the set of simple closed geodesics on $X$ and $\tau(\gamma)$ the trace of the holonomy of $\gamma \in \cS$.

It is natural to ask whether for any $m \in \cM$, the geodesic $\gamma$ such that $\tau(\gamma) = 3m$ is unique up to an isometry of $X$. It was proved by Schmutz \cite{Schmutz} that this statement is equivalent to \MUC{}.

\subsection{Acknowledgments.}
We thank Ser Peow Tan and David Dumas for valuable feedback.

\bigskip
\section{Proof of the theorem}

Greg McShane showed that, for any finite-volume hyperbolic metric on the punctured torus $T^*$,
\begin{equation}
\label{eq:McShane}
\sum_{\gamma\in\cS} \frac1{1+e^{\ell(\gamma)}} = \frac12
\end{equation}
where $\cS$ is the set of simple closed geodesics and $\ell(\gamma)$ indicates the length of $\gamma$ \cite{McShane}. Recalling that the trace and length of $\gamma$ are related by $\tau(\gamma) = 2\cosh(\ell(\gamma)/2)$, McShane's identity can be rewritten
\begin{align} \label{eq:Massage}
1&=\sum_\gamma \, \frac2{1+e^{\ell(\gamma)}} = \sum_\gamma \, e^{-\ell(\gamma)/2}\sech(\ell(\gamma)/2) \\
\nonumber
&= \sum_\gamma \, \frac2{\tau(\gamma) + \sqrt{\tau(\gamma)^2-4}}\cdot \frac2{\tau(\gamma)}
= \sum_\gamma \, 1- \sqrt{1-\frac4{\tau(\gamma)^2}}\,.
\end{align}
When $T^*$ with its hyperbolic metric is chosen to be the modular torus $X$, let us denote
$m(\gamma) \coloneqq \tau(\gamma)/3$ the associated Markov number (see \autoref{subsec:RelationModularTorus}) and $L(\gamma) \coloneqq \sqrt{9 - \frac{4}{m(\gamma)^2}}$ the associated Lagrange number. 
Reworking \eqref{eq:Massage}, McShane's identity on the modular torus is simply rewritten:
\begin{equation} \label{eq:McShaneRe}
 \sum_{\gamma \in \cS} \, (3 - L(\gamma)) \,= 3\,.
\end{equation}

It remains to investigate the fibers of the map $\gamma\mapsto L(\gamma)$ from simple closed geodesics on $X$ to Lagrange numbers. 
It is not hard to show that all fibers are nonempty: this is because
Vieta involutions act transitively on the Markov tree, and act as mapping classes on $\cS$.
By Schmutz's theorem \cite{Schmutz}, \MUC{} is equivalent to each fiber of $\gamma\mapsto L(\gamma)$ being the $\Aut(X)$-orbit of a single simple closed geodesic on $X$.
To finish the proof of \autoref{thm:MainThm}, we just need to count the number of elements of each orbit.

\begin{lemma} \label{lem:Orbits}
Let $\cS_0\subset \cS$ indicate the six shortest geodesics on $X$, and let $\cS_1=\cS - \cS_0$. Each orbit $\Aut(X)\curvearrowright \cS_0$ has three elements, and each orbit of $\Aut(X)\curvearrowright \cS_1$ has six elements.
\end{lemma}

\begin{proof}
There is an $\Aut(X)$-equivariant correspondence of $\cS$ with lines in $H:=H_1(X,\bZ)$. The standard generators $a, b$ of $\pi_1(X)\approx \pi_1(T^*)$ (as in \autoref{subsec:RelationModularTorus}) provide a basis of $H\approx \bZ^2$. 
The image of the homomorphism $\Aut(X)\to \PGL(2,\bZ)$ is the dihedral group with six elements, generated by
\begin{equation}
r=\begin{pmatrix}
0 & 1\\
-1 & -1
\end{pmatrix}
\text{ and }
\sigma=\begin{pmatrix}
0 & 1 \\
1 & 0
\end{pmatrix}\,.
\end{equation}
The actions of $r$ and $\sigma$ on $\bP^1 H$ have fixed points $\Fix(r)=\emptyset$ and $\Fix(\sigma)=\{[1:1],[1:-1]\}$. This implies that all simple closed geodesics on $X$ have six images under the action of $\Aut(X)$, except for the two geodesics corresponding to $ab$ and $ab^{-1}$, which have three such images apiece. 
These six geodesics are precisely the six shortest geodesics on $X$.
\end{proof}

Let us now prove \autoref{thm:MainThm}, in fact the slightly more precise version:

\begin{theorem} \label{thm:MainThmPrec}
We have $\displaystyle \sum_{n=1}^\infty (3-L_n) \leqslant 4 - \varphi - \sqrt{2}$, with equality if and only if \MUC{} holds.
\end{theorem}

\begin{proof}
Recall that $X$ denotes the modular torus and $\cS$ the set of simple closed geodesics on $X$. Let $\cS / \Aut(X)$ indicate the set of $\Aut(X)$-orbits in $\cS$. By \eqref{eq:McShaneRe}, the McShane identity on $X$ is rewritten:
\begin{equation}
\sum_{\gamma \in \cS} \, (3-L(\gamma)) = \sum_{A\in\cS / \Aut(X)} \ \sum_{\gamma\in A} \ 3-L(\gamma) \, = \, 3\,.
\end{equation}
By \autoref{lem:Orbits}, the map $\gamma\mapsto \tau(\gamma)$ is $6$-to-$1$ for $\gamma\in\cS_1$ and $3$-to-$1$ for $\gamma\in \cS_0$. 
Therefore, we get
\begin{equation}
 \left(6\sum_{[\gamma]\in\cS_1/\Aut(X)}\, + \, 3\sum_{[\gamma]\in\cS_0/\Aut(X)} \right)  \, \left(3-L(\gamma)\right) \,=\, 3\,.
\end{equation}

The six curves in $\cS_0$ are the shortest geodesics in $\cS$, so the two Lagrange numbers they determine are the two smallest Lagrange numbers $L_1=\sqrt 5$ and $L_2=\sqrt 8$. The previous equality can be written
\begin{equation}
\left(6 \sum_{[\gamma] \in \cS / \Aut(X)} \, (3-L(\gamma))\right) - 3\bigg((3 - L_1) + (3-L_2)\bigg) \, = \, 3\,,
\end{equation}
which we rewrite:
\begin{equation}
\sum_{[\gamma] \in \cS / \Aut(X)} \, (3-L(\gamma)) \, = \, 4 - \varphi - \sqrt{2}\,.
\end{equation}

The map $[\gamma] \mapsto L(\gamma)$ from $\cS / \Aut(X)$ to the set of Lagrange numbers $\cL = \{L_n, \, n \in \bN\}$ is onto, and one-to-one if and only if \MUC{} holds (see discussion above \autoref{lem:Orbits}). The conclusion follows.
\end{proof}

\section{Numerical evidence} \label{sec:Numerical}

Numerical computation suggests that the series $\sum_{n=1}^\infty (3-L_n)$ indeed converges to $L = 4 - \varphi - \sqrt{2}$. Denoting $R_n \coloneqq L - \sum_{k=1}^n (3-L_k)$ the presumed remainder, we find for instance $R_n \approx 7.34169 \times 10^{-455}$ for $n=50\, 000$. 

\begin{remark}
 Of course, one can also check \MUC{} directly with an algorithm (see \eg{} \cite{Metz}). A short Python script took us less than a minute on a personal computer to check \MUC{} for all Markov numbers $m_n$ up to $10^{1000}$, \ie up to $n =  959\,047$. Nevertheless, it is nice to get a different confirmation.
\end{remark}

Pushing the analysis further, we obtain new numerical evidence of Zagier's estimate
$m_n \sim \frac{1}{3} e^{C \sqrt{n}}$ where $C = 2.3523414972...$\,. Let us recall that this estimate is still open but was proved in weaker forms in \cite{Zagier} and \cite{MR1340065}. Elementary calculus involving the comparison of the remainder $R_n$ with the integral $6 \int_n^{+\infty} e^{-2C \sqrt{t}} \upd t$ translates Zagier's estimate to $R_n \sim \frac{6 \sqrt{n}}{C} e^{-2C \sqrt{n}}$. On \autoref{fig:Plots} it appears that the graph of $R_n$ in Log scale is indeed asymptotic to the expected curve. 

\begin{remark}[Computer code]
We wrote a simple recursive algorithm in Python to generate the list of Markov numbers.
We then used Mathematica to compute the remainders $R_n$ up to $n=  50\,000$ and plot the graphs. Our code is freely available on GitHub \cite{GitHub}.
\end{remark}

\begin{figure}[!ht]
\centering
\begin{subfigure}{.5\textwidth}
  \centering
  \includegraphics[width=.98\textwidth]{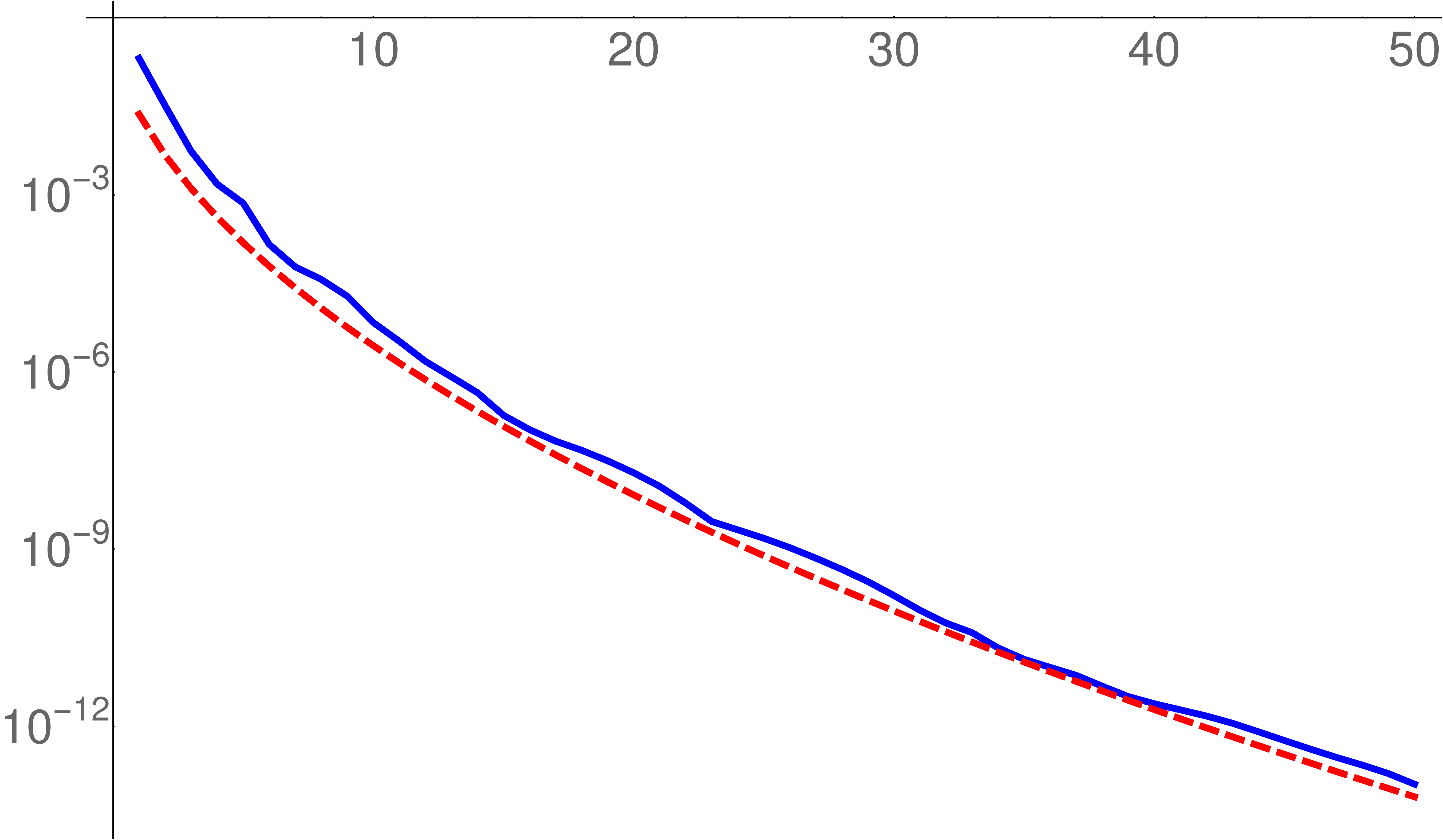}
  \caption{$n = 1 \dots 50$.}
  \label{fig:PlotA}
\end{subfigure}%
\begin{subfigure}{.5\textwidth}
  \centering
  \includegraphics[width=.98\textwidth]{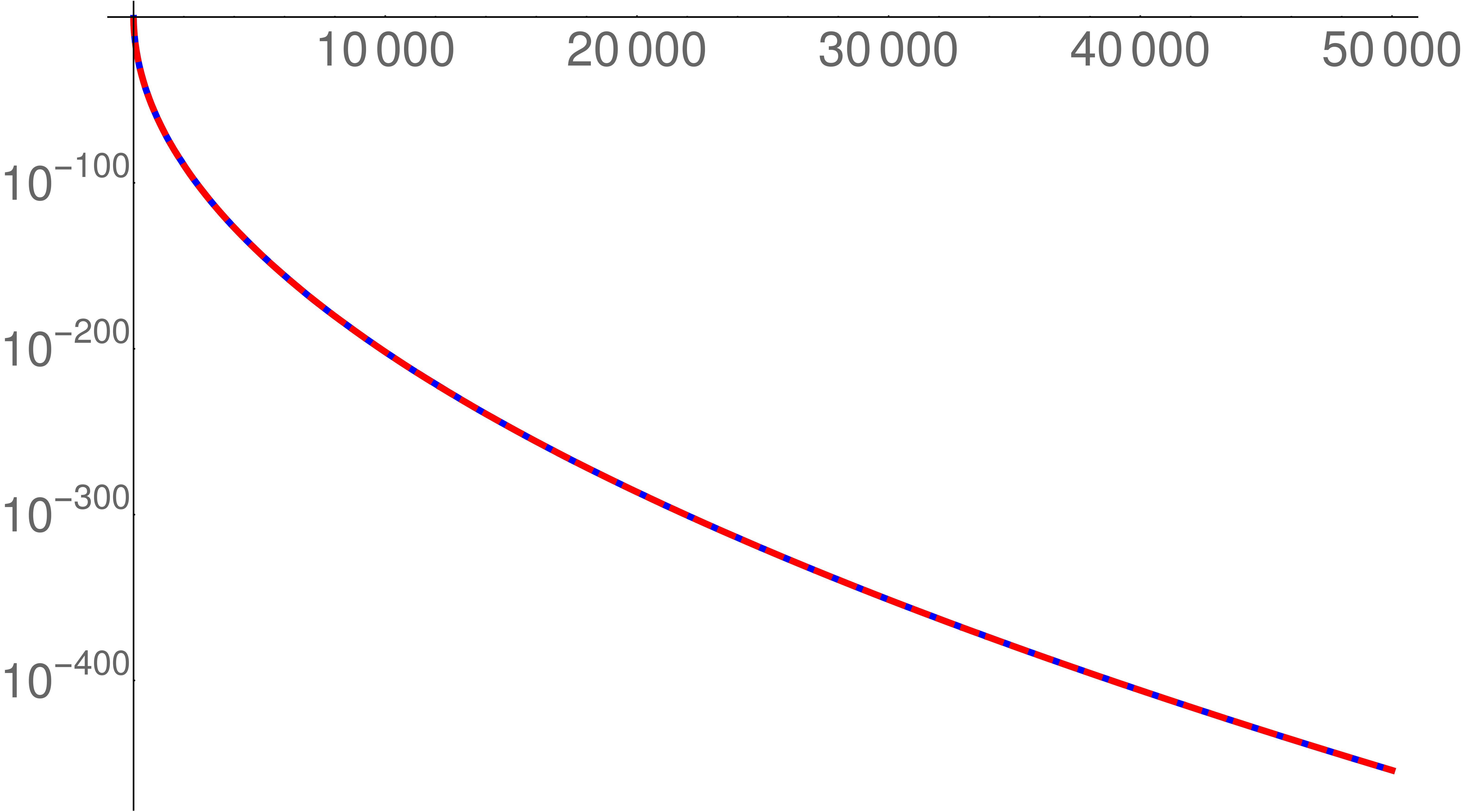}
  \caption{$n = 1 \dots 50\,000$.}
  \label{fig:PlotB}
\end{subfigure}
\caption{Numerical computation of the remainder $R_n = (4 - \varphi - \sqrt 2) - \sum_{k=1}^n (3-L_k)$.
The dashed curve shows the expected asymptotic profile $\frac{6 \sqrt{n}}{C} e^{-2C \sqrt{n}}$.}
\label{fig:Plots}
\end{figure}

\bigskip

\bibliographystyle{alpha}

\begin{thebibliography}{Bow96}

\bibitem[Aig15]{Aigner}
Martin Aigner.
\newblock {\em Markov's theorem and 100 years of the uniqueness conjecture}.
\newblock Springer, 2015.

\bibitem[Bow96]{MR1356829}
B.~H. Bowditch.
\newblock A proof of {M}c{S}hane's identity via {M}arkoff triples.
\newblock {\em Bull. London Math. Soc.}, 28(1):73--78, 1996.

\bibitem[CF89]{cusick-flahive}
Thomas~W Cusick and Mary~E Flahive.
\newblock {\em The {M}arkoff and {L}agrange spectra}.
\newblock Number~30. American Mathematical Soc., 1989.

\bibitem[Coh55]{MR67935}
Harvey Cohn.
\newblock Approach to {M}arkoff's minimal forms through modular functions.
\newblock {\em Ann. of Math. (2)}, 61:1--12, 1955.

\bibitem[Fro13]{frobenius}
Georg~Ferdinand Frobenius.
\newblock {\em {\"U}ber die Markoffschen Zahlen}.
\newblock K{\"o}nigliche Akademie der Wissenschaften, 1913.

\bibitem[Gol03]{Goldman}
William~M Goldman.
\newblock The modular group action on real {SL}(2)--characters of a one-holed
  torus.
\newblock {\em Geometry \& Topology}, 7(1):443--486, 2003.

\bibitem[Gor81]{Gorshkov}
D.~S. Gorshkov.
\newblock Geometry of {L}obachevskii in connection with certain questions of
  arithmetic.
\newblock {\em Journal of Soviet Mathematics}, 16:788--820, 1981.

\bibitem[Guy83]{Guy}
Richard~K Guy.
\newblock Don't try to solve these problems!
\newblock {\em The American Mathematical Monthly}, 90(1):35--41, 1983.

\bibitem[js20]{GitHub}
jbgaster and seub.
\newblock {GitHub} repository: {LagrangeSeries}, 2020.
\newblock URL: \url{https://github.com/seub/LagrangeSeries}.

\bibitem[LT07]{Lang-Tan}
Mong~Lung Lang and Ser~Peow Tan.
\newblock A simple proof of the {M}arkoff conjecture for prime powers.
\newblock {\em Geometriae Dedicata}, 129(1):15--22, 2007.

\bibitem[Mar79]{Markoff}
Andrey Markoff.
\newblock Sur les formes quadratiques binaires ind{\'e}finies.
\newblock {\em Mathematische Annalen}, 15(3-4):381--406, 1879.

\bibitem[Mar80]{MR1510073}
Andrey Markoff.
\newblock Sur les formes quadratiques binaires ind\'{e}finies {II}.
\newblock {\em Mathematische Annalen}, 17(3):379--399, 1880.

\bibitem[McS98]{McShane}
Greg McShane.
\newblock Simple geodesics and a series constant over {T}eichm\"uller space.
\newblock {\em Inventiones Mathematicae}, 132(3):607--632, 1998.

\bibitem[Met15]{Metz}
Brandon~John Metz.
\newblock {\em A Comparison of Recent Results on the Unicity Conjecture of the
  {M}arkoff Equation}.
\newblock Master thesis in {M}athematics, {U}niversity of {N}evada, {L}as
  {V}egas, 2015.
\newblock URL:
  \url{https://digitalscholarship.unlv.edu/cgi/viewcontent.cgi?article=3390&context=thesesdissertations}.

\bibitem[MR95]{MR1340065}
Greg McShane and Igor Rivin.
\newblock Simple curves on hyperbolic tori.
\newblock {\em C. R. Acad. Sci. Paris S\'{e}r. I Math.}, 320(12):1523--1528,
  1995.

\bibitem[Sch96]{Schmutz}
Paul Schmutz.
\newblock Systoles of arithmetic surfaces and the markoff spectrum.
\newblock {\em Mathematische Annalen}, 305(1):191--203, 1996.

\bibitem[Ser85]{Series}
Caroline Series.
\newblock The geometry of {M}arkoff numbers.
\newblock {\em The Mathematical Intelligencer}, 7(3):20--29, 1985.

\bibitem[Zag82]{Zagier}
Don Zagier.
\newblock On the number of {M}arkoff numbers below a given bound.
\newblock {\em Mathematics of Computation}, 39(160):709--723, 1982.

\end{thebibliography}

\end{document}